\documentclass{amsart}
\usepackage[cp1250]{inputenc}

\newtheorem{theorem}{Theorem}[section]
\newtheorem{proposition}[theorem]{Proposition}
\newtheorem{lemma}[theorem]{Lemma}
\newtheorem{corollary}[theorem]{Corollary}
\newtheorem{definition}[theorem]{Definition}
\newtheorem{example}[theorem]{Example}
\newtheorem{remark}[theorem]{Remark}
\numberwithin{equation}{section}

\begin{document}

    \title[Character formulas for Feigin-Stoyanovsky's type subspaces]{Character formulas for Feigin-Stoyanovsky's type subspaces of standard $\mathfrak{sl}(3, \mathbb{C})^{\widetilde{}}$-modules}

    \author{Miroslav Jerković}

    \address{Faculty of Chemical Engineering and Technology, University of Zagreb, Marulićev trg 19, Zagreb, Croatia}

    \curraddr{}

    \email{mjerkov@fkit.hr}

    \thanks{}

    \subjclass[2000]{Primary 17B67; Secondary 17B69, 05A19.\\ \indent Partially supported by the Ministry of Science and Technology of the Republic of Croatia, Project ID 037-0372794-2806}


    \date{}

    \dedicatory{}

    \begin{abstract} Exact sequences of Feigin-Stoyanovsky's type subspaces for affine Lie algebra $\mathfrak{sl}(l+1,\mathbb{C})^{\widetilde{}}$ lead to systems of recurrence relations for formal characters of those subspaces. By solving the corresponding system for $\mathfrak{sl}(3,\mathbb{C})^{\widetilde{}}$, we obtain a new family of character formulas for all Feigin-Stoyanovsky's type subspaces at general level. \end{abstract}

    \maketitle

\section{Introduction}

    In this paper we continue our study (cf. \cite{J1, J2}) of Feigin-Stoyanovsky's type subspaces, a family of distinguished subspaces of standard $\mathfrak{sl}(\ell + 1, \mathbb{C})^{\widetilde{}}$-modules. By solving the system of recursive equations for characters of these subspaces in the particular case of Feigin-Stoyanovsky's type subspaces of standard $\mathfrak{sl}(3, \mathbb{C})^{\widetilde{}}$-modules, we obtain a new class of fermionic-type formulas. The technique used in proving that these formulas present the solution of the corresponding system is somewhat similar to the one Andrews used in \cite{A} to solve Rogers-Selberg recursions, but more involved.
    
    We also show that, using appropriate specializations, these formulas reduce to some already known formulas (cf. \cite{FJMMT, FJMMT2}). In those papers the authors use different approaches in calculating character formulas for Feigin-Stoyanovsky's type subspaces. For an overview of these and other results in the area, as well as in studying similar family of so-called principal subspaces, see \cite{J1, J2}, or directly \cite{FS, G, CLM1, CLM2, Ca1, Ca2, FJMMT, FJMMT2}.

    The paper is organized as follows. Section 2 introduces the necessary setting for the subject of interest and gives a definition of a Feigin-Stoyanovsky's type subspace. Furthermore, Theorem \ref{thm_basis} on combinatorial basis consisting of so-called admissible monomial vectors is presented, as well as Theorem \ref{thm_exactness} on exact sequences connecting subspaces at general integer level. Section 3 contains the definition of formal character of a Feigin-Stoyanovsky's type subspace, and presents the system of recursive equations for characters of subspaces at given general integer level. This section also contains the main result, Theorem \ref{thm_solution}, which gives the fermionic-type formulas for Feigin-Stoyanovsky's type subspaces for standard $\mathfrak{sl}(3, \mathbb{C})^{\widetilde{}}$-modules, together with comparison of these formulas with already existing results. Finally, the last section presents the proof of the main results in a series of technical yet elementary lemmas.

    I sincerely thank Mirko Primc for his valuable suggestions.

\section{Feigin-Stoyanovsky's type subspaces}

\subsection{Preliminaries}

    Let $\mathfrak{g}$ be a simple Lie algebra $\mathfrak{sl}(\ell+1,\mathbb{C})$, $\mathfrak{h}$ a Cartan subalgebra, $R$ the root system with fixed simple roots $\alpha_1, \dots, \alpha_{\ell}$, and $\langle \cdot,\cdot\rangle$ the Killing form. We have the root space decomposition $\mathfrak{g} = \mathfrak{h} + \sum_{\alpha \in R} \mathfrak{g}_{\alpha}$, with fixed root vectors $x_{\alpha}$. Denote by $Q=Q(R)$ the root lattice and by $P=P(R)$ the weight lattice of $\mathfrak{g}$, fundamental weights being $\omega_1, \dots, \omega_{\ell}$.

    Let $\tilde{\mathfrak{g}}$ be the associated affine Lie algebra \begin{align*} \tilde{\mathfrak{g}} = \mathfrak{g} \otimes \mathbb{C}[t,t^{-1}] \oplus \mathbb{C} c \oplus \mathbb{C} d, \end{align*} c denoting the canonical central element and d the degree operator, with Lie product given in the usual way (cf. \cite{Kac}). We write $x(n) = x\otimes t^n$ for $x \in \mathfrak{g}$, $n \in \mathbb{Z}$.

    Denote by $\Lambda_0, \dots, \Lambda_{\ell}$ the corresponding fundamental weights of $\tilde{\mathfrak{g}}$. For given integral dominant weight $\Lambda = k_0\Lambda_0 + k_1 \Lambda_1 + \cdots + k_{\ell} \Lambda_{\ell}$, denote by $L(\Lambda)$ the standard $\tilde{\mathfrak{g}}$-module with highest weight $\Lambda$, and by $v_{\Lambda}$ a fixed highest weight vector of $L(\Lambda)$. Let $k=\Lambda(c)=k_0+\cdots + k_{\ell}$ be the level of $L(\Lambda)$.

\subsection{Definition of Feigin-Stoyanovsky's type subspaces}

    For fixed minuscule weight $\omega = \omega_{\ell}$ define \begin{align*}\Gamma= \{ \alpha \in R \mid \langle\alpha, \omega \rangle =1 \} = \{ \gamma_1, \gamma_2, \dots , \gamma_{\ell} \mid \gamma_i = \alpha_i + \dots + \alpha_{\ell} \}.\end{align*} Setting $\mathfrak{g}_0 = \mathfrak{h} + \sum_{\langle\alpha, \omega \rangle=0} \mathfrak{g}_{\alpha}$, $\mathfrak{g}_{\pm 1} = \sum_{\alpha \in \pm \Gamma} \mathfrak{g}_{\alpha}$, we have the following $\mathbb{Z}$-grading: \begin{align*} \mathfrak{g} = \mathfrak{g}_{-1} + \mathfrak{g}_0 + \mathfrak{g}_1,\end{align*} which gives us the corresponding $\mathbb{Z}$-grading of $\tilde{\mathfrak{g}}$: \begin{align*} \tilde{\mathfrak{g}} = \tilde{\mathfrak{g}}_{-1} + \tilde{\mathfrak{g}}_0 + \tilde{\mathfrak{g}}_1, \end{align*} with $\tilde{\mathfrak{g}}_0 = \mathfrak{g}_0 \otimes \mathbb{C}[t,t^{-1}]\oplus \mathbb{C} c \oplus \mathbb{C}d$, $ \tilde{\mathfrak{g}}_{\pm 1} = \mathfrak{g}_{\pm 1}\otimes\mathbb{C}[t,t^{-1}]$. Note that \begin{align*}\tilde{\mathfrak{g}}_1= \textrm{span} \{ x_{\gamma} (n) \mid \gamma \in \Gamma, n \in \mathbb{Z} \}\end{align*} is a commutative subalgebra and a $\tilde{\mathfrak{g}}_0$-module.

    \begin{definition} Feigin-Stoyanovsky's type subspace $W(\Lambda)$ of a standard $\tilde{\mathfrak{g}}$-module $L(\Lambda)$ is defined by \begin{align*} W(\Lambda) = U(\tilde{\mathfrak{g}}_1) \cdot v_{\Lambda},\end{align*} where $U(\tilde{\mathfrak{g}}_1)$ denotes the universal enveloping algebra of $\tilde{\mathfrak{g}}_1$. \end{definition}

\subsection{Combinatorial bases}

    Let us now fix level $k$ and fix an integral dominant highest weight $\Lambda = k_0\Lambda_0 + k_1 \Lambda_1 + \cdots + k_{\ell} \Lambda_{\ell}$ such that $k_0 + \dots + k_{\ell} = k$. Poincar\'{e}-Birkhoff-Witt theorem implies that $W(\Lambda)$ is spanned by the following set of monomial vectors: \begin{align} \label{span_eq1} \{ x(\pi)v_{\Lambda} | x(\pi)= \dots x_{\gamma_1}(-2)^{a_{\ell}} x_{\gamma_{\ell}}(-1)^{a_{\ell-1}} \cdots x_{\gamma_1}(-1)^{a_0}, a_i \in \mathbb{Z}_{+}, i \in \mathbb{Z}_{+}\}.\end{align} Primc in \cite{P1} shows how \eqref{span_eq1} can be reduced to a basis of $W(\Lambda)$ consisting of so-called $(k, \ell + 1)$-admissible monomial vectors:

    \begin{definition} A monomial vector from \eqref{span_eq1} is $(k, \ell + 1)$-admissible for $\Lambda$ if it satisfies difference conditions \begin{align*} a_i + \dots + a_{i+\ell} \leq k, \quad i \in \mathbb{Z}_{+} \end{align*} and initial conditions \begin{align*} a_0 &\leq k_0 \\ a_0+a_1 &\leq k_0 + k_1 \\ \dots \\ a_0 + a_1 + \dots + a_{\ell-1} &\leq k_0 + k_1 + \dots + k_{\ell-1}.\end{align*} \end{definition}

    \begin{theorem} \label{thm_basis} The set of $(k, \ell + 1)$-admissible monomial vectors for $\Lambda$ is a basis of $W(\Lambda)$. \end{theorem}

\subsection{Exact sequences}

    In \cite{J1} we use Theorem \ref{thm_basis}, as well as certain intertwining operators to obtain exact sequences of Feigin-Stoyanovsky's type subspaces for standard modules at arbitrary integer level, as well as accompanying systems of recurrence relations for their formal characters. Also, \cite{J1} relies on known Frenkel-Kac-Segal construction (see \cite{FK, S}, or \cite{FLM}) of fundamental $\tilde{\mathfrak{g}}$-modules $L(\Lambda_i)$ on the tensor product $M(1)\otimes \mathbb{C}[P]$ of the Fock space $M(1)$ for the homogeneous Heisenberg subalgebra and the group algebra $\mathbb{C}[P]$ of the weight lattice with a basis $e^{\lambda}$, $\lambda \in P$.

    We provide a short exposition of the results obtained in \cite{J1}. For fixed $\Lambda = k_0 \Lambda_0 + \dots + k_{\ell} \Lambda_{\ell}$ such that $k_0+\dots + k_{\ell} = k$, $k_i \in \mathbb{Z}_{+}$, $i = 0, \dots, \ell$, denote $W=W_{k_0,k_1,\dots, k_{\ell}} = W(\Lambda)$, and by $v$ highest weight vector of $L(\Lambda)$. Set $m=\sharp \{ i=0,\dots, \ell-1 \mid k_i \neq 0 \}$ and for $t=0, \dots, m-1$ define \begin{align*} W_{I_t} = W_{k_0, \dots, k_{i_0}-1, k_{i_0+1}+1, \dots, k_{i_{t-1}}-1, k_{i_{t-1} + 1}+1, \dots, k_{\ell}},\end{align*} where \begin{align} \label{indices} I_t = \{ \{ i_0, \dots, i_{t-1} \} | 0 \leq i_0 \leq \dots \leq i_{t-1} \leq \ell -1, k_{i_j} \neq 0, j = 0, \dots, t-1 \}.\end{align} By $v_{I_t}$ denote the corresponding highest weight vector.

    Define $U(\tilde{\mathfrak{g}}_1)-$homogeneous mappings $\varphi_t: \sum_{I_t} W_{I_t} \to \sum_{I_{t+1}} W_{I_{t+1}}$ by describing its action on corresponding highest weight vectors: \begin{align*} \varphi_t(v_{I_t}) = \sum_{\genfrac{}{}{0pt}{}{i, k_i \neq 0}{i \notin I_t}} (-1)^{\sharp \{ j \in I_t | j < i\}} v_{I_t \cup \{ i \}}.\end{align*}

    We also use the simple current operator, a linear bijection $[\omega]$ such that \begin{align*} & L(\Lambda_0) \xrightarrow{[\omega]} L(\Lambda_{\ell}) \xrightarrow{[\omega]} L(\Lambda_{\ell-1}) \xrightarrow{[\omega]} \dots \xrightarrow{[\omega]} L(\Lambda_1) \xrightarrow{[\omega]} L(\Lambda_0) \\ & [\omega] v_{\Lambda_0} = v_{\Lambda_{\ell}}, \quad [\omega] v_{\Lambda_i} = x_{\gamma_i}(-1)v_{\Lambda_{i-1}}, \quad i=1,\dots, \ell. \end{align*}

    \begin{theorem} \label{thm_exactness} The following sequence is exact: \begin{align*} 0 \rightarrow W_{k_{\ell},k_0,k_1, \dots, k_{\ell-1}} \xrightarrow{[\omega]^{\otimes k}} W \xrightarrow{\varphi_0} \sum_{I_1} W_{I_1} \xrightarrow{\varphi_1} \dots \xrightarrow{\varphi_{m-1}} W_{I_m} \rightarrow 0. \end{align*} \end{theorem}

    \begin{example} \label{ex_exactness_l2k2} These are exact sequences for $\ell = 2$ and $k=2$: \begin{align*} & 0 \to W_{0,2,0} \to W_{2,0,0} \to W_{1,1,0} \to 0 \\ & 0 \to W_{0,1,1} \to W_{1,1,0} \to W_{0,2,0} \oplus W_{1,0,1} \to W_{0,1,1} \to 0 \\ & 0 \to W_{1,1,0} \to W_{1,0,1} \to W_{0,1,1} \to 0 \\* & 0 \to W_{0,0,2} \to W_{0,2,0} \to W_{0,1,1} \to 0 \\ & 0 \to W_{1,0,1} \to W_{0,1,1} \to W_{0,0,2} \to 0 \\ & 0 \to W_{2,0,0} \to W_{0,0,2} \to 0 \end{align*} \end{example}

\section{Formal characters of Feigin-Stoyanovsky's type subspaces}

\subsection{Definition of formal character}

    \begin{definition} \label{def_degree} For $x(\pi)= \dots x_{\gamma_1}(-2)^{a_{\ell}} x_{\gamma_{\ell}}(-1)^{a_{\ell-1}} \cdots x_{\gamma_1}(-1)^{a_0}$ define degree $d(x(\pi))$ and weight $w(x(\pi))$ of $x(\pi)$: \begin{align*} d(x(\pi)) = \sum_{j=0}^{\infty} \sum_{i = 1}^{\ell} (j+1) a_{i + j \ell -1}, \quad w(x(\pi)) = \sum_{j=0}^{\infty} \sum_{i = 1}^{\ell} \gamma_i a_{i + j \ell -1}.\end{align*} \end{definition}

    \begin{definition} Formal character of $W = W(\Lambda)$ is given by \begin{align}\label{character} \chi(W)(z_1,\dots, z_{\ell}; q)= \sum \dim W^{m,n_1,\dots,n_{\ell}}q^m z_1^{n_1}\cdots z_{\ell}^{n_{\ell}},\end{align} with $W^{m,n_1,\dots,n_{\ell}}$ denoting the component of $W$ spanned by basis monomial vectors $x(\pi) v$ such that $d(x(\pi)) = m$ and $w(x(\pi)) = n_1 \gamma_1 + \dots + n_{\ell} \gamma_{\ell}$. \end{definition}

\subsection{Recurrence relations for formal characters}

    Theorem \ref{thm_exactness} produces a system of recursive equations connecting formal characters of all Feigin-Stoyanovsky's type subspaces of standard $\mathfrak{sl}(\ell + 1, \mathbb{C})^{\widetilde{}}$-modules at arbitrary integer level $k$: \begin{align} \label{system} & \sum_{I} (-1)^{|I|} \chi(W_I)(z_1,\dots, z_{\ell};q) = (z_1q)^{k_0} \dots (z_{\ell}q)^{k_{\ell-1}} \chi(W_{k_{\ell},k_0,\dots, k_{\ell-1}})(z_1q,\dots, z_{\ell}q;q), \end{align} where we sum over all $I$ given by \eqref{indices}. Introducing \begin{align*} \chi(W_{k_0,\dots,k_{\ell}})(z_1,\dots, z_{\ell};q)= \sum_{n_1,\dots, n_{\ell} \geq 0} A_{k_0,\dots, k_{\ell}}^{n_1,\dots, n_{\ell}}(q)z_1^{n_1}\dots z_{\ell}^{n_{\ell}}\end{align*} into \eqref{system} gives \begin{align} \label{A_system} \sum_{I} (-1)^{|I|} A_{I}^{n_1, \dots, n_{\ell}}(q) = q^{n_1 + \dots + n_{\ell}} A_{k_{\ell},k_0,\dots, k_{\ell-1}}^{n_1-k_0, \dots, n_{\ell}-k_{\ell-1}}(q).\end{align} Specially, for $\ell = 2$ the system \eqref{A_system} consists of all equations of the following form (noting that if any of the lower indices is less than zero, set the corresponding summand to vanish): \begin{align} \label{A_system_l2} & A_{k_0,k_1,k_2}^{n_1,n_2}(q) - A_{k_0-1,k_1+1,k_2}^{n_1,n_2}(q) - A_{k_0,k_1-1,k_2+1}^{n_1,n_2}(q) + A_{k_0-1,k_1,k_2+1}^{n_1,n_2}(q) = \\ \nonumber & = q^{n_1+n_2}A_{k_2,k_0,k_1}^{n_1-k_0,n_2-k_1}(q).\end{align} We say that a specific equation of \eqref{A_system_l2} is indexed by triple $(k_0,k_1,k_2)$ if the first term appearing in it is $A_{k_0,k_1,k_2}^{n_1,n_2}(q)$.

    \begin{example} The following system is obtained from Example \ref{ex_exactness_l2k2}: \begin{align} \label{ex_A_system_l2k2} & A_{2,0,0}^{n_1,n_2}(q) - A_{1,1,0}^{n_1,n_2}(q) = q^{n_1+n_2}A_{0,2,0}^{n_1-2,n_2}(q) \\ \nonumber & A_{1,1,0}^{n_1,n_2}(q) - A_{0,2,0}^{n_1,n_2}(q) - A_{1,0,1}^{n_1,n_2}(q) + A_{0,1,1}^{n_1,n_2}(q) = q^{n_1+n_2}A_{0,1,1}^{n_1-1,n_2-1}(q) \\ \nonumber & A_{1,0,1}^{n_1,n_2}(q) - A_{0,1,1}^{n_1,n_2}(q) = q^{n_1+n_2}A_{1,1,0}^{n_1-1,n_2}(q) \\ \nonumber & A_{0,2,0}^{n_1,n_2}(q) - A_{0,1,1}^{n_1,n_2}(q) = q^{n_1+n_2}A_{0,0,2}^{n_1,n_2-2}(q) \\ \nonumber & A_{0,1,1}^{n_1,n_2}(q) - A_{0,0,2}^{n_1,n_2}(q) = q^{n_1+n_2}A_{1,0,1}^{n_1,n_2-1}(q) \\ \nonumber & A_{0,0,2}^{n_1,n_2}(q) = q^{n_1+n_2}A_{2,0,0}^{n_1,n_2}(q). \end{align} \end{example}

\subsection{Character formulas}

    We solved \eqref{A_system_l2} for $\ell = 2$ and $k$ general positive integer, thus obtaining fermionic-type formulas for all Feigin-Stoyanovsky's type subspaces of all standard $\mathfrak{sl}(3, \mathbb{C})^{\widetilde{}}$-modules. Before stating this result we first introduce the appropriate notation.

    \begin{definition} Let $\leq$ be partial ordering defined on $\{0, 1 \}^k$ in the following manner: given $p_1 = (p_{1,1},\dots, p_{1,k})$ and $p_2 = (p_{2,1},\dots, p_{2,k})$ in $\{0, 1 \}^k$ write $p_1 \leq p_2$ if $\sum_{i=1}^j p_{1,i} \geq \sum_{i=1}^j p_{2,i}$ holds for all $j=1, \dots, k$. \end{definition}

    \begin{definition} \label{def_linear} Given $p = (p_1, \dots, p_k) \in \{ 0,1 \}^k$ and two sets of integers $N_{1,1} \geq N_{1,2} \geq \dots \geq N_{1,k} \geq 0$ and $N_{2,k} \geq N_{2,k-1} \geq \dots \geq N_{2,1} \geq 0$, define \begin{align*} l_p^1 (q) = q^{\sum_{i=1}^k p_i N_{1,i}},\quad \delta_p^1(q) =\prod_{i=1}^k (1- \delta_{p_i - p_{i+1},-1} q^{N_{1,i} - N_{1,i+1}}) \\ l_p^2 (q) = q^{\sum_{i=1}^k p_i N_{2,i}}, \quad \delta_p^2(q) =\prod_{i=1}^k (1- \delta_{p_i - p_{i-1},-1} q^{N_{2,i} - N_{2,i-1}}). \end{align*} Set also $N_{1,k+1}=0$ and $N_{2,0}=0$ and note that we let $p_0 = 0$ if we want that $1-q^{N_{2,1}}$ never appears in $\delta_p^2(q)$ or $p_0 = 1$ if we want $1-q^{N_{2,1}}$ to appear in $\delta_p^2(q)$ when $p_1=0$. Similarly, we write $p_{k+1}=0$ if we don't want $1-q^{N_{1,k}}$ to be part of $\delta_p^1(q)$ or $p_{k+1}=1$ if we permit the appearance of $1-q^{N_{1,k}}$ in $\delta_p^1(q)$, in case of $p_k=0$. Unless stated otherwise, we presume that $p_0 = 0$ and $p_{k+1}=0$. \end{definition}

    \begin{remark} Note that although $N_{1,i}, N_{2,i}$, $i =1, \dots, k$, are arguments of above defined functions, we omit them in order to be concise in notation (it will be clear from context which are their values). \end{remark}

    \begin{definition} For $i=0, \dots, k$, define $ \mathcal{P}_{i}^k = \{ (p_1, \dots, p_k) \in \{ 0,1 \}^k | \sum_{j=1}^k p_j = i \},$ i.e. $\mathcal{P}_{i}^k$ contains $k$-tuples from $\{ 0,1 \}^k$ consisting of $i$ ones and $k-i$ zeros. \end{definition}

    \begin{definition} Given $p = (p_1, \dots, p_k) \in \{ 0,1 \}^k$, $i=1,\dots, k$, and $j=1,2$, define the following $k$-tuples by describing their coordinate components for $m=1,\dots, k$: \begin{align*} & (f_{i}^j (p))_m = \left\{ \begin{array}{l} 1-j, \quad p_{m} = j \quad \textrm{and} \quad \sharp \{ p_{n} | p_{n} =j, n \leq m \} \leq i \\ p_{m}, \qquad  \textrm{otherwise} \end{array}\right. \\ & (g_{i}^j (p))_m = \left\{ \begin{array}{l} 1-j, \quad p_{m} = j \quad \textrm{and} \quad \sharp \{ p_{n} | p_{n} =j, n \geq m \} \leq i \\ p_{m}, \qquad \textrm{otherwise} \end{array}\right..\end{align*} \end{definition}

    \begin{remark} It is obvious that these $k$-tuples again belong to $\{ 0,1\}^k$, and are obtained from $p$ by changing first (or last) $i$ ones of $p$ into zeros, or vice versa. Note that $f_i^0(p)$ and $g_i^0(p)$ make sense only if $p \in \mathcal{P}_{j}^{k}$ for some $j \leq k-i$ ($k$-tuples with at least $i$ zeros), while $f_i^1(p)$ and $g_i^1(p)$ are well defined only if $p \in \mathcal{P}_{j}^{k}$ for some $j \geq i$ ($k$-tuples with at least $i$ ones). \end{remark}

    \begin{definition} For $p \in \mathcal{P}_{i}^k$ and $j =1, \dots, i$, define $pos_{1,j}(p)$ by $p_{pos_{1,j}(p)} = 1$ and $\sharp \{ p_n | p_n = 1, n \leq pos_{1,j}(p) \} = j$, i.e. $pos_{1,j}(p)$ defines the coordinate position of $j$-th one in $p$. Also, for $j =1, \dots, k-i$ define $pos_{0,j}(p)$, the coordinate position of $j$-th zero in $p$, by $p_{pos_{0,j}(p)} = 0$ and $\sharp \{ p_n | p_n = 0, n \leq pos_{0,j}(p) \} = j$. \end{definition}

    We have now set ground to state the main result of this paper.

    \begin{theorem} \label{thm_solution} Let $\Lambda = k_0\Lambda_0 + k_1\Lambda_1 + k_2\Lambda_2$ be the highest weight of the level $k$ standard $\mathfrak{sl}(3, \mathbb{C})^{\widetilde{}}$-module $L(\Lambda)$. For the character $\chi(W(\Lambda))(z_1,z_2; q)$ of Feigin-Stoyanovsky's type subspace $W(\Lambda)$ the following formula holds: \begin{align} & \label{thm_solution_eq1} \chi(W(\Lambda))(z_1,z_2; q) = \\ \nonumber & = \sum_{n_1,n_2 \geq 0} \sum_{ \genfrac{}{}{0pt}{}{\genfrac{}{}{0pt}{}{\sum_{i=1}^k N_{1,i} = n_1}{N_{1,1} \geq \cdots \geq N_{1,k} \geq 0}}{\genfrac{}{}{0pt}{}{\sum_{i=1}^k N_{2,i} = n_2}{N_{2,k} \geq \cdots \geq N_{2,1} \geq 0}}} {\frac{q^{\sum_{i=1}^k N_{1,i}^2 + N_{2,i}^2 + N_{1,i}N_{2,i}} L_{k_0,k_1,k_2}(q)} {\prod_{i=1}^k (q)_{N_{1,i}-N_{1,i+1}} \prod_{i=1}^k (q)_{N_{2,i} - N_{2,i-1}}}} z_1^{n_1}z_2^{n_2},\end{align} $L_{k_0,k_1,k_2}(q) = \sum_{p \in \mathcal{P}_{k_1+k_2}^k} l_p^1(q) \delta_{p}^1(q) l_{g_{k_1}^1(p)}^2(q)$ being the "linear" term of the formula. \end{theorem}

    \begin{example} \label{ex_k2} Let us present character formulas for Feigin-Stoyanovsky's type subspaces of level $2$ standard $\mathfrak{sl}(3, \mathbb{C})^{\widetilde{}}$-modules by first giving the "linear" terms: \begin{align*} L_{2,0,0}(q)&= 1 \\ L_{1,1,0}(q)&= q^{N_{1,1}} + q^{N_{1,2}}(1- q^{N_{1,1} - N_{1,2}}) = q^{N_{1,2}} \\ L_{1,0,1}(q)&=q^{N_{1,1}+N_{2,1}}+q^{N_{1,2} + N_{2,2}}(1-q^{N_{1,1}- N_{1,2}}) \\ L_{0,2,0}(q)&= q^{N_{1,1}+N_{1,2}} \\ L_{0,1,1}(q)&= q^{N_{1,1}+N_{1,2}+N_{2,1}} \\ L_{0,0,2}(q)&= q^{N_{1,1}+N_{1,2}+N_{2,1}+N_{2,2}}.\end{align*} With $L_{k_0,k_1,k_2}(q)$ as above, we now have: \begin{align} & \label{ex_k2_eq1} \chi(W(k_0\Lambda_0 + k_1\Lambda_1 + k_2\Lambda_2))(z_1,z_2; q) = \\ \nonumber & = \sum_{n_1,n_2 \geq 0} \sum_{ \genfrac{}{}{0pt}{}{\genfrac{}{}{0pt}{}{N_{1,1} + N_{1,2} = n_1}{N_{1,1} \geq N_{1,2} \geq 0}}{\genfrac{}{}{0pt}{}{ N_{2,1} + N_{2,2} = n_2}{N_{2,2} \geq N_{2,1} \geq 0}}} \frac{q^{N_{1,1}^2 + N_{1,2}^2 +N_{2,1}^2 + N_{2,2}^2 + N_{1,1}N_{2,1} + N_{1,2}N_{2,2} } L_{k_0,k_1,k_2}(q)}{(q)_{N_{1,1}-N_{1,2}}(q)_{N_{1,2}}(q)_{N_{2,2}-N_{2,1}}(q)_{N_{2,2}}} z_1^{n_1} z_2^{n_2}. \end{align} \end{example}

    \begin{remark} First note that similar remark holds for $L_{k_0,k_1,k_2}(q)$ as for functions defined in Definition \ref{def_linear}, so we omit writing $N_{1,1}$, $N_{1,2}$, $N_{2,1}$, $N_{2,2}$. Furthermore, in Theorem \ref{thm_solution}, as well as in Example \ref{ex_k2}, the linear term is organized in such way that factors from $\delta_{p}^1(q)$ appear (e.g. see $L_{1,0,1}(q)$). It can be shown, and will be of later use, that there is an equivalent way to describe the linear term using $\delta_{p}^2(q)$, with $p \in \mathcal{P}_{k_2}^k$ (cf. Corollary \ref{cor_equality}). \end{remark}

    We first prove the following technical result:

    \begin{lemma} \label{lem_smaller} For every $p \in \mathcal{P}_{i}^k$ the following holds: \begin{align} \label{lem_smaller_eq1} l_p^1(q) = \sum_{\genfrac{}{}{0pt}{}{p' \leq p}{p' \in \mathcal{P}_i^k}} l_{p'}^1(q) \delta_{p'}^1 (q). \end{align} \end{lemma}

    \begin{proof} Before we state the proof, note that there is an analogous statement for $l_p^2(q)$, proven in the same fashion as \eqref{lem_smaller_eq1}. We start by multiplying all factors on right-hand side of \eqref{lem_smaller_eq1}. If, for some $p'' \in \mathcal{P}_i^k$, $p'' \leq p$, we multiply $l_{p''}^1 (q)$ with one of the factors $\pm q^{N_{1,i} - N_{1,i+1}}$ appearing in $\delta_{p^{''}}^1 (q)$, we see that we obtain $\pm l_{p'}^1(q)$ for some $p' \in \mathcal{P}_i^k$, $p' \leq p$, since the number of zeros and ones did not change during this operation, and since it is obvious that $p' \leq p'' \leq p$.

    Let us now fix some $p' \in \mathcal{P}_i^k$, $p' \leq p$. It is easy to see that contributions (in the above sense) to $l_{p'}^1(q)$ come from all such $p'' \in \mathcal{P}_i^k$, $p'' \leq p$, that $p''_j = p'_j$ for all $j=1, \dots, k$, except maybe for some $i=1, \dots, k$ for which $p'_i =1$, $p'_{i+1}=0$ and $p''_{i}=0$, $p''_{i+1}=1$ holds. For those $p''$ we write $p'' \sim p'$ (note that $p' \sim p'$). If $p'' \sim p'$ is such that $p''_j = p''_j$ for all $j=1,\dots, k$, except for exactly one such $i$ for which $p'_i=1$, $p'_{i+1}=0$, $p''_i=0$, $p''_{i+1}=1$, then we write $p'' \sim_i p'$. Define $n(p') = \sharp \{ i | \exists p'' \sim_i p'\}$ and note that $n(p') > 0$ for all $p' < p$, while $p$ is the only element of $\mathcal{P}_i^k$ less or equal $p$ such that $n(p)=0$.

    If we now calculate the factor that $l_{p'}^1(q)$ comes with after all multiplications on the right-hand side of \eqref{lem_smaller_eq1} are done, we get \begin{align*} 1- \binom{n(p')}{1} + \binom{n(p')}{2} + \dots + (-1)^{n(p')} \binom{n(p')}{n(p')} = (1+(-1))^{n(p')},\end{align*} which equals zero for all $p' < p$, and one only in case of $p' = p$. This proves the assertion \eqref{lem_smaller_eq1}.\end{proof}

    \begin{proposition} \label{prop_interchange} Set $i, j = 0,\dots, k$ such that $i+j \leq k$. Then \begin{align} \label{prop_interchange_eq1} \sum_{p \in \mathcal{P}_{i+j}^k} l_{p}^1(q) \delta_{p}^1(q) l^2_{g_i^1(p)}(q) = \sum_{p' \in \mathcal{P}_{j}^k} l^1_{g_i^0(p')}(q) l_{p'}^2(q) \delta_{p'}^2(q). \end{align} \end{proposition}

    \begin{proof} Using Lemma \ref{lem_smaller} we calculate \begin{align*}& \sum_{p \in \mathcal{P}_{i+j}^k} l_{p}^1(q) \delta_{p}^1(q) l^2_{g_i^1(p)}(q) \overset{\eqref{lem_smaller_eq1}}{=} \sum_{p \in \mathcal{P}_{i+j}^k} l_{p}^1(q) \delta_{p}^1(q) \Big( \sum_{\genfrac{}{}{0pt}{}{p' \geq g_i^1(p)}{p' \in \mathcal{P}_j^k}} l_{p'}^2(q) \delta_{p'}^2 (q) \Big) \overset{(*)}{=} \\ & = \sum_{p' \in \mathcal{P}_{j}^k} \Big( \sum_{\genfrac{}{}{0pt}{}{p'' \leq g_i^0(p')}{p'' \in \mathcal{P}_{i+j}^k}} l^1_{p''}(q) \delta_{p''}^1(q) \Big) l_{p'}^2(q) \delta_{p'}^2(q) \overset{\eqref{lem_smaller_eq1}}{=} \sum_{p' \in \mathcal{P}_{j}^k} l^1_{g_i^0(p')}(q)l_{p'}^2(q) \delta_{p'}^2(q), \end{align*} where $(*)$ denotes the observation that for all $p \in \mathcal{P}_{i+j}^k$ and all $p' \in \mathcal{P}_{j}^k$ the condition $p' \geq g_i^1(p)$ covers all $p' \in \mathcal{P}_{j}^k$, and then that for fixed $p' \in \mathcal{P}_{j}^k$ all those $p'' \in \mathcal{P}_{i+j}^k$ such that $p' \geq g_i^1(p'')$ holds are exactly those for which $p'' \leq g_{i}^0(p')$ is true. \end{proof}

    The following consequence of Proposition \ref{prop_interchange} states that linear term $L_{k_0,k_1,k_2}(q)$ can be rewritten to be parameterized by $k$-tuples from $\mathcal{P}_{k_2}^k$:

    \begin{corollary} \label{cor_equality} By setting $i=k_1$ and $j=k_2$ in \eqref{prop_interchange_eq1} we get: \begin{align} \label{cor_equality_eq1} L_{k_0,k_1,k_2}(q) =\sum_{p \in \mathcal{P}_{k_1 + k_2}^k} l_{p}^1(q) \delta_{p}^1(q) l^2_{g_{k_1}^1(p)}(q) \overset{\eqref{prop_interchange_eq1}}{=} \sum_{p' \in \mathcal{P}_{k_2}^k} l^1_{g_{k_1}^0(p')}(q) l_{p'}^2(q) \delta_{p'}^2(q). \end{align} \end{corollary}

    Note finally that linear term \eqref{cor_equality_eq1} simplifies substantially for $k_0=0$ or $k_2=0$: \begin{align*} & k_0 =0: \mathcal{P}_{k_1 + k_2}^k = \mathcal{P}_k^k = \{ (1,\dots, 1) \} \Rightarrow L_{k_0,k_1,k_2}(q)= q^{N_{1,1} + \dots + N_{1,k} + N_{2,1} + \dots + N_{2,k_2}} \\ & k_2=0: P_{k_2}^k = P_0^k = \{ (0,\dots, 0) \} \Rightarrow L_{k_0,k_1,k_2}(q)= q^{N_{1,k_0 + 1} + \dots + N_{1,k}}.\end{align*}

\subsection{Comparison with known results}

    The authors of \cite{FJMMT} embed the dual space of principal subspace $W(\Lambda)$ for level $k$ standard $\mathfrak{sl}(3, \mathbb{C})^{\widetilde{}}$-modules into the space of symmetric polynomials, where they introduce the so-called Gordon filtration. By explicitly calculating components of the associated graded space (using vertex operators), they obtained principally specialized character formulas for $W(\Lambda)$ when $\Lambda= k_0 \Lambda_0 + k_1 \Lambda_1$, i.e. for $k_2=0$: \begin{align} \label{fjmmt_eq1} & \chi_{FJMMT}(W(k_0 \Lambda_0 + k_1 \Lambda_1))(z; q) = \\ \nonumber & \sum_{n \geq 0} \sum_{\genfrac{}{}{0pt}{}{l_1 +l_2 = n}{ l_1,l_2 \geq 0}} \sum_{\genfrac{}{}{0pt}{}{ \sum_j jm_{ij}=l_i}{ i=1,2}} \frac{q^{{}^t m A m - (\textrm{diag} A)\cdot m + 2c_{k_0}^{(3)}\cdot m}}{(q^2)_{m_{11}}\cdots (q^2)_{m_{1k}} (q^2)_{m_{21}}\cdots (q^2)_{m_{2k}}} q^{l_2} z^n, \end{align} where \begin{align*} A&= \left({\begin{array}{c|c} A^{(2)} & B^{(3)} \\ \hline B^{(3)} & A^{(2)}\end{array} }\right) \\ A^{(2)} &= (A_{ab}^{(2)})_{1\leq a,b \leq k}, \quad A_{ab}^{(2)}=2\textrm{min}(a,b) \\ B^{(3)} &= (B_{ab}^{(3)})_{1\leq a,b \leq k}, \quad B_{ab}^{(3)}=\textrm{max}(0,a+b-k) \\ c_{k_0}^{(3)} &=(\underbrace{0,\dots,0,1,2,\dots, k-k_0}_{k}, \underbrace{0,\dots,0}_{k}) \\ m&= {}^t(m_{11},\dots, m_{1k},m_{21},\dots, m_{2k}). \end{align*}

    \begin{example} For $k=2$ we get \begin{align*} A = \left( \begin{array}{cccc} 2 & 2 & 0 & 1 \\ 2 & 4 & 1 & 2 \\ 0 & 1 & 2 & 2 \\ 1 & 2 & 2 & 4 \end{array} \right), \end{align*} while $c_{k_0}^{(3)}$ gives three possible linear terms $L_{k_0,k_1}(m_{11},m_{12},m_{21},m_{22})$ in \eqref{fjmmt_eq1}: \begin{align*} \Lambda = 2\Lambda_0 & \quad \Rightarrow \quad L_{2,0}(m_{11},m_{12},m_{21},m_{22}) = -2 m_{11} - 4 m_{12} - m_{21} - 2 m_{22} \\ \Lambda = \Lambda_0 + \Lambda_1 & \quad \Rightarrow \quad L_{1,1}(m_{11},m_{12},m_{21},m_{22}) = -2 m_{11} - 2 m_{12} - m_{21} - 2 m_{22} \\ \Lambda = 2\Lambda_1 & \quad \Rightarrow \quad L_{0,2}(m_{11},m_{12},m_{21},m_{22}) = -m_{21} - 2 m_{22}. \end{align*} Quadratic term in all three cases is \begin{align*} & Q(m_{11}, m_{12}, m_{21}, m_{22}) = 2m_{11}^2 + 4m_{11}m_{12} + 4m_{12}^2 + \\ & + 2m_{12}m_{21} + 2m_{21}^2+ 2 m_{11}m_{22} + 4m_{12}m_{22} + 4m_{21}m_{22} + 4m_{22}^2, \end{align*} and we get the following formulas: \begin{align} \label{fjmmt_eq2} & \chi_{FJMMT}(W(k_0 \Lambda_0+ k_1 \Lambda_1))(z; q) = \\ \nonumber & = \sum_{n \geq 0} \sum_{\genfrac{}{}{0pt}{}{l_1 +l_2 = n}{ l_1,l_2 \geq 0}} \sum_{\genfrac{}{}{0pt}{}{ \sum_j jm_{ij}=l_i}{ i,j =1,2}} \frac{q^{Q(m_{11}, m_{12}, m_{21}, m_{22}) + L_{k_0, k_1}(m_{11}, m_{12}, m_{21}, m_{22})}}{(q^2)_{m_{11}}(q^2)_{m_{12}}(q^2)_{m_{21}}(q^2)_{m_{22}}} q^{l_2} z^n.\end{align} \end{example}

    It is not hard to demonstrate for $k=2$ and $k_2=0$ that specialized character \begin{align*} \chi_{spec_1}(W(k_0\Lambda_0 + k_1\Lambda_1))(z; q)\end{align*} obtained from \eqref{ex_k2_eq1} by applying suitable specialization \begin{align} \label{spec1_eq1}spec_1: \quad q \to q^2, \quad z_1 \to q^{-2}z, \quad z_2 \to q^{-1}z \end{align} becomes \eqref{fjmmt_eq2} (note that the same can be shown for general $k$). First, after introducing the following dependance among $N_{i,j}$, $i=1, \dots, k$, $j=1,2$, and $m_{11}, m_{12}, m_{21}, m_{22}$: \begin{align} \label{dependence_eq1} & m_{11} := N_{1,1}-N_{1,2}, \quad m_{12} := N_{1,2} \\ \nonumber & m_{21} := N_{2,2} - N_{2,1}, \quad m_{22} := N_{2,1}, \end{align} it turns out that the summation conditions for \eqref{ex_k2_eq1} and \eqref{fjmmt_eq2} are the same: \begin{align*} N_{1,1} + N_{1,2} = n_1, \quad N_{1,1} \geq N_{1,2} \geq 0 \quad & \Rightarrow \quad m_{11} + 2m_{12} = n_1, \quad m_{11},m_{12} \geq 0 \\ N_{2,1} + N_{2,2} = n_2, \quad N_{2,2} \geq N_{2,1} \geq 0 \quad &\Rightarrow \quad m_{21} + 2m_{22} = n_2, \quad m_{21},m_{22} \geq 0. \end{align*} Furthermore, it is obvious that denominator of a summand in \eqref{ex_k2_eq1} by applying \eqref{spec1_eq1} and \eqref{dependence_eq1} transforms into denominator of a corresponding summand in \eqref{fjmmt_eq2}. It is also easy to check that both quadratic and linear parts of \eqref{ex_k2_eq1} transform under \eqref{spec1_eq1} and \eqref{dependence_eq1} into the corresponding parts of \eqref{fjmmt_eq2}, therefore \begin{align*} \chi_{spec_1}(W(k_0\Lambda_0 + k_1\Lambda_1))(z;q) = \chi_{FJMMT}(W(k_0 \Lambda_0+ k_1 \Lambda_1))(z;q). \end{align*}

    Let us now present connection with the main result of \cite{FJMMT2}. The authors study finite $(k,3)$-admissible configurations, sequences of non-negative integers ${\bf a} = (a_i)_{i\in{\bf Z}_{\geq0}}$ satisfying the conditions $a_i + a_{i+1} + a_{i+2} \leq k$ and such that $a_i =0$ for $i \gg 0$. They construct a bijection between the set of such sequences and the set of so-called rigged partitions, consequently obtaining a certain fermionic formula for $\sum_{{\bf{a}} \in C^{(k,l)}_{a,b}[N]} q^{d(\bf{a})}$, with $d({\bf a})=\sum_{i=0}^\infty ia_i$ being the degree of ${\bf a}$, and $C^{(k,l)}_{a,b}[N]$ denoting the set of finite $(k,3)$-admissible configurations of maximal weight $l$ (cf. $(2.4)$ and $(2.5)$ in \cite{FJMMT2}), satisfying initial conditions $a_0 =a$, $a_1 = b$ for $0 \leq a,b \leq k$, and such that $a_i =0$ for $i > N$ (cf. Theorem 4.7 in \cite{FJMMT2}). It is easy to see that the set of $(k,3)$-admissible configurations satisfying initial conditions for given $\Lambda = k_0 \Lambda_0 + k_1 \Lambda_1 + k_2\Lambda_2$ equals to \begin{align*} \bigcup_{\genfrac{}{}{0pt}{}{a \leq k_0}{a+b \leq k_0 + k_1}}C^{(k,k)}_{a,b}[\infty].\end{align*} Since the definition of $d({\bf a})$ differs from one presented in Definition \ref{def_degree}, and since the authors of \cite{FJMMT2} are not interested in the concept of weight, in order to compare the characters we introduce the following specialization: \begin{align*} spec_2: \quad q \to q^2, \quad z_1 \to q^{-2}, \quad z_2 \to q^{-1},\end{align*} which applied to $\chi(W(k_0\Lambda_0 + k_1\Lambda_1 + k_2\Lambda_2))(z_1,z_2; q)$ in \eqref{character} gives \begin{align}\label{fjmmt2_eq1} & \chi_{spec_2}(W(k_0\Lambda_0 + k_1\Lambda_1 + k_2\Lambda_2))(q) = \\ \nonumber & = \sum_{\genfrac{}{}{0pt}{}{a \leq k_0}{a+b \leq k_0 + k_1}} \sum_{{\bf a}\in C^{(k,k)}_{a,b}[\infty]} q^{d(\bf{a})} = \\ \nonumber & = \chi^{(k,k)}_{0, k_0 + k_1}[\infty] + \chi^{(k,k)}_{1, k_0 + k_1 - 1}[\infty] + \dots + \chi^{(k,k)}_{k_0, k_1}[\infty] - \\ \nonumber & \quad - (\chi^{(k,k)}_{0, k_0+k_1+1}[\infty] + \chi^{(k,k)}_{1, k_0 + k_1}[\infty] + \dots + \chi^{(k,k)}_{k_0 - 1, k_1 + 2}[\infty]), \end{align} with \begin{align*} \chi^{(k,k)}_{a, b}[N] = \sum_{m_1, \dots, m_k =0}^{\infty} q^{Q({\bf m}) + \sum_{i=1}^k r_i m_i} \prod_{\genfrac{}{}{0pt}{}{1 \leq j \leq k}{m_j \neq 0}} \begin{bmatrix} jN - \sum_{i=1}^k A_{ji} m_i+ A_{jj} - r_j + m_j \\ m_j \end{bmatrix},\end{align*} \begin{align*} Q({\bf m}) &= \frac{1}{2} (A {\bf m},{\bf m}) - \frac{1}{2} \sum_{j=1}^{k} A_{jj} m_j \\ {\bf m} & = (m_1, \dots, m_k) \\ A_{ij} & = 2\min(i,j) + \max(i+j-k,0) \\ (r_1, \dots, r_k) & = (\underbrace{0,\dots,0}_{a},\underbrace{1,\dots,b}_{b},\underbrace{b+2,\dots, 2k - 2a - b}_{k-a-b}) \\ \begin{bmatrix} m \\ n \end{bmatrix} &= \left\{ \begin{array}{l} \prod_{i=1}^n \frac{1-q^{m-n+1}}{1-q^i} \quad 0 \leq n \leq m \\ 0, \quad \textrm{otherwise} \end{array}\right.. \end{align*}

    Relation \eqref{fjmmt2_eq1} substantially simplifies for $k_0=0$ or $k_2=0$ to \begin{align} \label{fjmmt2_eq2} \chi_{spec_2}(W(k_0\Lambda_0 + k_1\Lambda_1 + k_2\Lambda_2))(q) = \chi^{(k,k)}_{k_0,k_1}[\infty],\end{align} i.e. \eqref{fjmmt2_eq1} doesn't simplify to \eqref{fjmmt2_eq2} if both $k_0 \neq 0$ and $k_2 \neq 0$, which is exactly when linear term of $\chi(W(k_0\Lambda_0 + k_1\Lambda_1 + k_2\Lambda_2))(z_1, z_2; q)$ consists of multiple summands.

\section{Proof of the main result}

    Theorem \ref{thm_solution} is proven by directly checking that formulas \eqref{thm_solution_eq1} for $\chi(W(k_0\Lambda_0 + k_1\Lambda_1 + k_2\Lambda_2))(z_1,z_2; q)$ satisfy the system \eqref{system} when $\ell = 2$ is introduced; more precisely, by checking that corresponding \begin{align} \label{solution_eq1} A_{k_0, k_1, k_2}^{n_1, n_2}(q) = \sum_{ \genfrac{}{}{0pt}{}{\genfrac{}{}{0pt}{}{\sum_{i=1}^k N_{1,i} = n_1}{N_{1,1} \geq \cdots \geq N_{1,k} \geq 0}}{\genfrac{}{}{0pt}{}{\sum_{i=1}^k N_{2,i} = n_2}{N_{2,k} \geq \cdots \geq N_{2,1} \geq 0}}} {\frac{q^{\sum_{i=1}^k N_{1,i}^2 + N_{2,i}^2 + N_{1,i}N_{2,i}} L_{k_0,k_1,k_2}(q)} {\prod_{i=1}^k (q)_{N_{1,i}-N_{1,i+1}} \prod_{i=1}^k (q)_{N_{2,i} - N_{2,i-1}}}}, \end{align} with $L_{k_0,k_1,k_2}(q)$ as in \eqref{thm_solution_eq1}, satisfy the system \eqref{A_system_l2}. We focus on proving that \eqref{solution_eq1} satisfy the most complex equations of that system - those indexed by strictly positive triples $(k_0, k_1, k_2)$, e.g. second equation appearing in \eqref{ex_A_system_l2k2}. The reasoning in checking other equalities of \eqref{A_system_l2} is similar, but simpler.

    The proof is organized in a series of technical lemmas which taken together produce given result (cf. Corollary \ref{cor_proof}). First we investigate the interplay between linear terms of summands appearing at the left-hand side of \eqref{A_system_l2}:

    \begin{lemma} \label{lem_0}For linear term as defined in \eqref{thm_solution_eq1} the following holds: \begin{align*} & L_{k_0,k_1,k_2}(q) - L_{k_0-1,k_1+1,k_2}(q) - L_{k_0,k_1-1,k_2+1}(q) + L_{k_0-1,k_1,k_2+1}(q) = L_{k_0,k_1,k_2}^{\star}(q), \end{align*} where $L_{k_0,k_1,k_2}^{\star}(q) = L_{k_0,k_1,k_2}(q) (1- q^{N_{2, pos_{1, k_2+1}(p)}})$, but with $p_{k+1}= 1$, $p \in \mathcal{P}_{k_1+k_2}^k$. \end{lemma}

    \begin{proof} By using \eqref{thm_solution_eq1} we get \begin{align} \label{lem_0_eq1} L_{k_0,k_1,k_2}(q) - L_{k_0,k_1-1,k_2+1}(q) & = \sum_{p \in \mathcal{P}_{k_1+k_2}^k} l_p^1(q) \delta_{p}^1(q) \Big( l_{g_{k_1}^1(p)}^2(q) - l_{g_{k_1-1}^1(p)}^2(q) \Big) \\ \nonumber L_{k_0-1,k_1+1,k_2}(q) - L_{k_0-1,k_1,k_2+1}(q) & = \sum_{p \in \mathcal{P}_{k_1+k_2 +1}^k} l_p^1(q) \delta_{p}^1(q) \Big( l_{g_{k_1 + 1}^1(p)}^2(q) - l_{g_{k_1}^1(p)}^2(q) \Big).\end{align} On $\mathcal{P}_{k_1+k_2+1}^k$ introduce relation $\sim$ that for given $p'$ and $p''$ states $p' \sim p''$ if $pos_{1,i}(p') = pos_{1,i}(p'')$, $i=1, \dots, k_1+k_2$. As class representatives we take $p'$ with greatest possible position of last one: $pos_{1,k_1+k_2 + 1}(p')= k$. Also, if given such $p'$ by $p$ we denote an element of $\mathcal{P}_{k_1+k_2}^k$ such that $pos_{1,i}(p) = pos_{1,i}(p')$ holds for all $i=1,\dots, k_1 + k_2$, then $p' \mapsto p$ gives a bijective correspondence between $p' \in \mathcal{P}_{k_1+k_2+1}^k$ with $pos_{1,k_1+k_2+1}(p') = k$ and $p \in \mathcal{P}_{k_1+k_2}^k$ having $pos_{1,k_1+k_2}(p) < k$. It is not hard to see that for given class representative $p'$ and $p$ as above the following is true: \begin{align} \label{lem_0_eq2} \sum_{p'' \sim p'} l_{p''}^1(q)\delta_{p''}^1(q) = l_{p}^1(q)\delta_{p}^1(q)q^{N_{1,k}}, \quad l_{g_{k_1+1}^1(p'')}^2(q) = l_{g_{k_1}^1(p)}^2(q).\end{align} Using also the obvious fact \begin{align}\label{lem_0_eq3} l_{g_{k_1}^1(p)}^2(q) - l_{g_{k_1-1}^1(p)}^2(q) = l_{g_{k_1}^1(p)}^2(q) (1-q^{N_{2, pos_{1, k_2+1}(p)}})\end{align}we now have for second equation in \eqref{lem_0_eq1}: \begin{align} \label{lem_0_eq4} & L_{k_0-1,k_1+1,k_2}(q) - L_{k_0-1,k_1,k_2+1}(q) \overset{\eqref{lem_0_eq1}}{=} \\ \nonumber & \overset{\eqref{lem_0_eq1}}{=} \sum_{p \in \mathcal{P}_{k_1+k_2 +1}^k} l_p^1(q) \delta_{p}^1(q) \Big( l_{g_{k_1 + 1}^1(p)}^2(q) - l_{g_{k_1}^1(p)}^2(q) \Big) = \\ \nonumber & = \sum_{\genfrac{}{}{0pt}{}{p' \in \mathcal{P}_{k_1+k_2+1}^k}{pos_{1,k_1+k_2+1}(p')=k}} \sum_{p'' \sim p'} l_{p''}^1(q) \delta_{p''}^1(q) \Big( l_{g_{k_1 + 1}^1(p'')}^2(q) - l_{g_{k_1}^1(p'')}^2(q) \Big) \overset{\eqref{lem_0_eq2}}{\underset{\eqref{lem_0_eq3}}{=}} \\ \nonumber & \overset{\eqref{lem_0_eq2}}{\underset{\eqref{lem_0_eq3}}{=}} \sum_{\genfrac{}{}{0pt}{}{p \in \mathcal{P}_{k_1+k_2}^k}{pos_{1,k_1+k_2}(p) < k}} l_{p}^1(q) \delta_{p}^1(q) q^{N_{1,k}} l_{g_{k_1}^1 (p)}^2(q) (1-q^{N_{2,pos_{1,k_2+1}(p)}}), \end{align} and finally \begin{align*} & L_{k_0,k_1,k_2}(q) - L_{k_0-1,k_1+1,k_2}(q) - L_{k_0,k_1-1,k_2+1}(q) + L_{k_0-1,k_1,k_2+1}(q) \overset{\eqref{lem_0_eq1}}{\underset{\eqref{lem_0_eq4}}{=}} \\ & \overset{\eqref{lem_0_eq1}}{\underset{\eqref{lem_0_eq4}}{=}} \sum_{p \in \mathcal{P}_{k_1+k_2}^k} l_p^1(q) \delta_{p}^1(q) l_{g_{k_1}^1 (p)}^2(q) (1-q^{N_{2,pos_{1,k_2+1}(p)}}) - \\ & \quad - \sum_{\genfrac{}{}{0pt}{}{p \in \mathcal{P}_{k_1+k_2}^k}{pos_{1,k_1+k_2}(p) < k}} l_{p}^1(q) \delta_{p}^1(q) q^{N_{1,k}} l_{g_{k_1}^1 (p)}^2(q) (1-q^{N_{2,pos_{1,k_2+1}(p)}}) = \\ & = \sum_{\genfrac{}{}{0pt}{}{p \in \mathcal{P}_{k_1+k_2}^k}{pos_{1,k_1+k_2}(p) = k}} l_{p}^1(q) \delta_{p}^1(q) l_{g_{k_1}^1 (p)}^2(q) (1-q^{N_{2,pos_{1,k_2+1}(p)}}) + \\ & \quad + \sum_{\genfrac{}{}{0pt}{}{p \in \mathcal{P}_{k_1+k_2}^k}{pos_{1,k_1+k_2}(p) < k}} l_{p}^1(q) \delta_{p}^1(q)(1-q^{N_{1,k}}) l_{g_{k_1}^1 (p)}^2(q) (1-q^{N_{2,pos_{1,k_2+1}(p)}}) = \\ &= L_{k_0,k_1,k_2}^{\star}(q). \end{align*} \end{proof}

    \begin{remark} Note that corresponding summands of $L_{k_0,k_1,k_2}(q)$ and $L_{k_0,k_1,k_2}^{\star}(q)$ differ only in the fact that $L_{k_0,k_1,k_2}^{\star}(q)$ can also contain, besides factors appearing in $L_{k_0,k_1,k_2}(q)$, factors $1 - q^{N_{1,k}}$ (for $p \in \mathcal{P}_{k_1 + k_2}^k$ such that $p_k =0$) and $1- q^{N_{2, pos_{1, k_2+1}(p)}}$ (if $k_2 < k$). \end{remark}

    We continue by looking at $q^{n_1+n_2}A_{k_2,k_0,k_1}^{n_1-k_0,n_2-k_1}(q)$, appearing on the right-hand side of \eqref{A_system_l2} which, after having introduced an appropriate formula from \eqref{solution_eq1}, has the following form: \begin{align} \label{solution_eq2} & q^{n_1+n_2}A_{k_2,k_0,k_1}^{n_1 - k_0, n_2 - k_1}(q) = \\ \nonumber & = q^{n_1+n_2}\sum_{ \genfrac{}{}{0pt}{}{\genfrac{}{}{0pt}{}{\sum_{i=1}^k N_{1,i} = n_1 - k_0}{N_{1,1} \geq \cdots \geq N_{1,k} \geq 0}}{\genfrac{}{}{0pt}{}{\sum_{i=1}^k N_{2,i} = n_2 - k_1}{N_{2,k} \geq \cdots \geq N_{2,1} \geq 0}}} \frac{q^{\sum_{i=1}^k N_{1,i}^2 + N_{2,i}^2 + N_{1,i}N_{2,i}}L_{k_2,k_0,k_1}(q)}{\prod_{i=1}^k (q)_{N_{1,i}-N_{1,i+1}} \prod_{i=1}^k (q)_{N_{2,i} - N_{2,i-1}}}.\end{align} Note that here we use version of the linear term appearing on the right-hand side of \eqref{cor_equality_eq1}: $ L_{k_2,k_0,k_1}(q) = \sum_{p \in \mathcal{P}_{k_1}^k} l^1_{g_{k_0}^0(p)}(q) l_{p}^2(q) \delta_{p}^2(q)$.

    \begin{lemma} \label{lem_1} \begin{align*} & q^{n_1+n_2}A_{k_2,k_0,k_1}^{n_1 - k_0, n_2 - k_1}(q) = q^{n_1} \sum_{ \genfrac{}{}{0pt}{}{\genfrac{}{}{0pt}{}{\sum_{i=1}^k N_{1,i} = n_1 - k_0}{N_{1,1} \geq \cdots \geq N_{1,k} \geq 0}}{\genfrac{}{}{0pt}{}{\sum_{i=1}^k N_{2,i} = n_2}{N_{2,k} \geq \cdots \geq N_{2,1} \geq 0}}} \frac{q^{\sum_{i=1}^k N_{1,i}^2 + N_{2,i}^2 + N_{1,i}N_{2,i}} M_{k_0,k_1,k_2}(q)}{\prod_{i=1}^k (q)_{N_{1,i}-N_{1,i+1}} \prod_{i=1}^k (q)_{N_{2,i} - N_{2,i-1}}}, \end{align*} with $M_{k_0,k_1,k_2}(q) = \sum_{p' \in \mathcal{P}_{k_0+k_2}^k} l_{f^1_{k_2}(p')}^1(q) l_{p'}^2(q) \delta_{p'}^2(q)$, stating $p'_0 :=1$. \end{lemma}

    \begin{proof} For each $p = (p_1, \dots, p_k) \in \mathcal{P}_{k_1}^k$ we transform the corresponding summand on the right-hand side of \eqref{solution_eq2} using the following transformation of the summation indices: \begin{align} \label{lem_1_eq1} N_{2,i} \to N_{2,i} - p_i, \quad i = 1, \dots, k.\end{align} After applying \eqref{lem_1_eq1} to $\sum_{i=1}^k N_{2,i} = n_2 - k_1$ we see that, because of $\sum_{i=1}^k p_i = k_1$, the "usual" summation condition for $N_{2,1}, \dots, N_{2,k}$ now holds: $\sum_{i=1}^k N_{2,i} = n_2$. On the other hand, monotony condition $N_{2,k} \geq \cdots \geq N_{2,1} \geq 0$ gets disturbed, but we will tract this problem a bit later.

    We transform the quadratic term together with the summand of the linear term corresponding to $p$, as well as with $q^{n_2}$, appearing in front of the sum sign on the right-hand side of \eqref{solution_eq2}: \begin{align}\label{lem_1_eq2} & k_1 + \sum_{i=1}^k N_{1,i}^2 + N_{2,i}^2 + N_{1,i}N_{2,i} + g_{k_0}^0(p)_i N_{1,i} + p_i N_{2,i} + N_{2,i} \xrightarrow{\eqref{lem_1_eq1}} \\ \nonumber & \xrightarrow{\eqref{lem_1_eq1}} \sum_{i=1}^k N_{1,i}^2 + N_{2,i}^2 + N_{1,i}N_{2,i} + (g_{k_0}^0(p)_i -p_i) N_{1,i} + (1 - p_i) N_{2,i} = \\ \nonumber &= \sum_{i=1}^k N_{1,i}^2 + N_{2,i}^2 + N_{1,i}N_{2,i} + f_{k_2}^1(p')_i N_{1,i} + p'_i N_{2,i}, \end{align} where for $p' \in \mathcal{P}_{k_0+k_2}^k$ defined by $p':=1-p$ it is easy to see that $g_{k_0}^0(p) - p = f_{k_2}^1(p')$. Observe that quadratic term we started with did not itself change at the end, it only contributed to the change in the linear part.

    Transformation of individual factors appearing in $\delta_p^2(q)$ of $L_{k_2,k_0,k_1}(q)$ will be realized together with transformation of the corresponding term in the denominator of the right-hand side of \eqref{solution_eq2}, i.e. we look at how $\frac{1-\delta_{p_{i} - p_{i-1},-1}q^{N_{2,i} - N_{2,i-1}}}{(q)_{N_{2,i} - N_{2,i-1}}}$ transforms under \eqref{lem_1_eq1}. This is also a good place to discuss the change in the monotony condition $N_{2,k} \geq \cdots \geq N_{2,1} \geq 0$. We have the following distinct possibilities:

    \begin{itemize}
    \item[a)] $p_{i} = 0, p_{i-1}=0 \Rightarrow \delta_{p_{i} - p_{i-1},-1}=0$, so corresponding factor of $\delta_p^2(q)$ equals one. Also, no change occurs in $N_{2,i}$ nor $N_{2,i-1}$, so we have \begin{align*} \frac{1}{(q)_{N_{2,i} - N_{2,i-1}}} \xrightarrow{\eqref{lem_1_eq1}} \frac{1}{(q)_{N_{2,i} - N_{2,i-1}}}. \end{align*} Summation condition $N_{2,i} \geq N_{2,i-1}$ also remains the same.
    \item[b)] $p_{i} = 1, p_{i-1}=1 \Rightarrow \delta_{p_{i} - p_{i-1},-1}=0$: although we here have $ N_{2,i} \to N_{2,i}-1$ and $N_{2,i-1}\to N_{2,i-1}-1$, these transformations cancel each other out, and again we get no change nor in factor of $\delta_p^2(q)$: \begin{align*} \frac{1}{(q)_{N_{2,i} - N_{2,i-1}}} \xrightarrow{\eqref{lem_1_eq1}} \frac{1}{(q)_{N_{2,i} - N_{2,i-1}}}, \end{align*} nor in the summation condition $N_{2,i} \geq N_{2,i-1}$.
    \item[c)] $p_{i} = 0, p_{i-1}=1 \Rightarrow \delta_{p_{i} - p_{i-1},-1}=1$, so the corresponding factor in $\delta_p^2(q)$ equals $1-q^{N_{2,i}-N_{2,i-1}}$, and with change $ N_{2,i} \to N_{2,i}$, $N_{2,i-1}\to N_{2,i-1}-1$, we get \begin{align*} \frac{1-q^{N_{2,i}-N_{2,i-1}}}{(q)_{N_{2,i} - N_{2,i-1}}} \xrightarrow{\eqref{lem_1_eq1}} \frac{1-q^{N_{2,i}-N_{2,i-1}+1}}{(q)_{N_{2,i} - N_{2,i-1}+1}} = \left\{\begin{array}{r} 0, N_{2,i} = N_{2,i-1} - 1 \\ \frac{1}{(q)_{N_{2,i} - N_{2,i-1}}}, N_{2,i} \geq N_{2,i-1} \end{array}\right. .\end{align*} On the other hand, monotony condition $N_{2,i} \geq N_{2,i-1}$ by \eqref{lem_1_eq1} changes to $N_{2,i} \geq N_{2,i-1} - 1$. But, as stated above, since the transformed factor equals zero in the case $N_{2,i} = N_{2,i-1} - 1$, we may assume that the monotony condition has not been disturbed by this transformation.
    \item[d)] $p_{i} = 1, p_{i-1}=0 \Rightarrow \delta_{p_{i} - p_{i-1},-1}=0$, and $N_{2,i} \to N_{2,i} - 1$, $N_{2,i-1} \to N_{2,i-1}$. Summation condition $N_{2,i} \geq N_{2,i-1}$ transforms to $N_{2,i} \geq N_{2,i-1} + 1$. We get: \begin{align*} \frac{1}{(q)_{N_{2,i} - N_{2,i-1}}} \xrightarrow{\eqref{lem_1_eq1}} \frac{1}{(q)_{N_{2,i} - N_{2,i-1}-1}} = \frac{1 - q^{N_{2,i} - N_{2,i-1}}}{(q)_{N_{2,i} - N_{2,i-1}}},\end{align*} where the last equality holds for $N_{2,i} \geq N_{2,i-1}$. But, since for $N_{2,i} = N_{2,i-1}$ transformed factor equals zero, we may extend the transformed summation condition $N_{2,i} \geq N_{2,i-1} + 1$ to include also this case, i.e. we may assume that the summation condition remained the same.
    \end{itemize}

    Note that the above discussion includes also the case $i=1$, which explains the appearance of the term $1-q^{N_{2,1}}$ when $p_{1}=1$, i.e. when $p_{1}'=0$.

    It is not hard to see that all the above possibilities can be gathered in this way: \begin{align} \label{lem_1_eq3} \frac{1-\delta_{p_{i} - p_{i-1},-1}q^{N_{2,i} - N_{2,i-1}}}{(q)_{N_{2,i} - N_{2,i-1}}} \xrightarrow{\eqref{lem_1_eq1}} \frac{1 - \delta_{p_{i} - p_{i-1}, 1} q^{N_{2,i} - N_{2,i-1}}}{(q)_{N_{2,i} - N_{2,i-1}}}, \end{align} with monotony summation condition remaining unchanged. After noting that $\delta_{p_{i} - p_{i-1}, 1} = \delta_{p_{i}' - p_{i-1}', -1}$, \eqref{lem_1_eq2} and \eqref{lem_1_eq3} prove the Lemma. \qedhere \end{proof}

    \begin{lemma} \label{lem_2} \begin{align} & M_{k_0,k_1,k_2}(q) = N_{k_0,k_1,k_2}(q), \end{align} where $N_{k_0,k_1,k_2}(q)=\sum_{p'' \in \mathcal{P}_{k_0}^k} l_{p''}^1(q) \delta_{p''}^1(q) l_{f_{k_2}^0(p'')}^2(q)(1-q^{N_{2,pos_{0,1}(f_{k_2}^0 (p''))}})$. \end{lemma}

    \begin{proof} Using Lemma \ref{lem_smaller} and similar reasoning as in Proposition \ref{prop_interchange} we get \begin{align*} & M_{k_0,k_1,k_2}(q) = \sum_{p' \in \mathcal{P}_{k_0+k_2}^k} l_{f^1_{k_2}(p')}^1(q) l_{p'}^2(q) \delta_{p'}^2(q) \overset{\eqref{lem_smaller_eq1}}{=} \\ & \overset{\eqref{lem_smaller_eq1}}{=} \sum_{p' \in \mathcal{P}_{k_0+k_2}^k} \Big( \sum_{p'' \leq f^1_{k_2}(p')} l_{p''}^1(q) \delta_{p''}^1(q) \Big) l_{p'}^2(q) \delta_{p'}^2(q) = \\ &= \sum_{p'' \in \mathcal{P}_{k_0}^k} l_{p''}^1(q) \delta_{p''}^1(q) \Big( \sum_{p' \geq f_{k_2}^0(p'')} l_{p'}^2(q) \delta_{p'}^2(q) \Big), \end{align*} so Lemma will be proved after showing \begin{align} \label{lem_2_eq1} \sum_{p' \geq f_{k_2}^0(p'')} l_{p'}^2(q) \delta_{p'}^2(q) = l_{f_{k_2}^0(p'')}^2(q)(1-q^{N_{2,pos_{0,1}(f_{k_2}^0 (p''))}}).\end{align} Although $p'$ are originally set to have $p'_0=1$ (cf. Lemma \ref{lem_1}), we now set $p'_0:=0$ (cf. Definition \ref{def_linear}). Also, denote by $f_{k_2}^0 (p'')_{min}$ the smallest element in $\mathcal{P}_{k_0 + k_2}^k$ greater than $f_{k_2}^0 (p'')$ and such that $(f_{k_2}^0 (p'')_{min})_1=0$. We now prove \eqref{lem_2_eq1}, omitting writing $p'_0 = 0$ after first line: \begin{align*}& \sum_{p' \geq f_{k_2}^0(p'')} l_{p'}^2(q) \delta_{p'}^2(q) = \sum_{\genfrac{}{}{0pt}{}{p' \geq f_{k_2}^0(p'')}{p'_0 =0, p'_1=0}} l_{p'}^2(q) \delta_{p'}^2(q) (1-q^{N_{2,1}})+ \sum_{\genfrac{}{}{0pt}{}{p' \geq f_{k_2}^0(p'')}{p'_0 =0, p'_1=1}} l_{p'}^2(q) \delta_{p'}^2(q) = \\ & = \sum_{p' \geq f_{k_2}^0(p'')} l_{p'}^2(q) \delta_{p'}^2(q) - q^{N_{2,1}} \sum_{\genfrac{}{}{0pt}{}{p' \geq f_{k_2}^0(p'')}{p'_1=0}} l_{p'}^2(q) \delta_{p'}^2(q) \overset{\eqref{lem_smaller_eq1}}{=} \\ & \overset{\eqref{lem_smaller_eq1}}{=} l_{f_{k_2}^0(p'')}^2(q) - q^{N_{2,1}} l_{f_{k_2}^0 (p'')_{min}}^2 (q) = l_{f_{k_2}^0(p'')}^2(q) - q^{N_{2,pos_{0,1}(f_{k_2}^0 (p''))}} l_{f_{k_2}^0(p'')}^2(q) = \\ & = l_{f_{k_2}^0(p'')}^2(q)(1-q^{N_{2,pos_{0,1}(f_{k_2}^0 (p''))}}). \end{align*} \qedhere \end{proof}

    \begin{lemma} \label{lem_3} \begin{align} \label{lem_3_eq1} & q^{n_1} \sum_{ \genfrac{}{}{0pt}{}{\genfrac{}{}{0pt}{}{\sum_{i=1}^k N_{1,i} = n_1 - k_0}{N_{1,1} \geq \cdots \geq N_{1,k} \geq 0}}{\genfrac{}{}{0pt}{}{\sum_{i=1}^k N_{2,i} = n_2}{N_{2,k} \geq \cdots \geq N_{2,1} \geq 0}}} \frac{q^{\sum_{i=1}^k N_{1,i}^2 + N_{2,i}^2 + N_{1,i}N_{2,i}} N_{k_0,k_1,k_2}(q)}{\prod_{i=1}^k (q)_{N_{1,i}-N_{1,i+1}} \prod_{i=1}^k (q)_{N_{2,i} - N_{2,i-1}}} = \\ \nonumber & = \sum_{ \genfrac{}{}{0pt}{}{\genfrac{}{}{0pt}{}{\sum_{i=1}^k N_{1,i} = n_1}{N_{1,1} \geq \cdots \geq N_{1,k} \geq 0}}{\genfrac{}{}{0pt}{}{\sum_{i=1}^k N_{2,i} = n_2}{N_{2,k} \geq \cdots \geq N_{2,1} \geq 0}}} \frac{q^{\sum_{i=1}^k N_{1,i}^2 + N_{2,i}^2 + N_{1,i}N_{2,i}} L_{k_0,k_1,k_2}^{\star}(q)}{\prod_{i=1}^k (q)_{N_{1,i}-N_{1,i+1}} \prod_{i=1}^k (q)_{N_{2,i} - N_{2,i-1}}}. \end{align} \end{lemma}

    \begin{proof} Proof is analogous to the one of Lemma \ref{lem_1}: for each given $p'' = (p''_1, \dots, p''_k)$ from $\mathcal{P}_{k_0}^k$ transform the corresponding summand on the left-hand side of \eqref{lem_3_eq1} using the following transformation of summation indices: \begin{align} \label{lem_3_eq2} N_{1,i} \to N_{1,i} - p''_{i}, \quad i = 1, \dots, k. \end{align} After applying \eqref{lem_3_eq2}, the summation condition $\sum_{i=1}^k N_{1,i} = n_1 - k_0$ now reads $\sum_{i=1}^k N_{1,i} = n_1$. Although the monotony condition N$_{1,1} \geq \cdots \geq N_{1,k} \geq 0$ gets disrupted, we will see that it "corrects" after taking into account change in the denominator part of the summand.

    Similarly as in proof of Lemma \ref{lem_1}, we see that transforming the quadratic part does not effect that term itself, but produces some extra terms to be added to the transformed linear part. We calculate this change taking into account also the term $q^{n_1}$, appearing outside the sum sign on the left-hand side of \eqref{lem_3_eq1}: \begin{align}\label{lem_3_eq3} & k_0 + \sum_{i=1}^k N_{1,i}^2 + N_{2,i}^2 + N_{1,i}N_{2,i} + p''_i N_{1,i} + f_{k_2}^0(p'')_i N_{2,i} + N_{1,i} \xrightarrow{\eqref{lem_3_eq2}} \\ \nonumber & \xrightarrow{\eqref{lem_3_eq2}} \sum_{i=1}^k N_{1,i}^2 + N_{2,i}^2 + N_{1,i}N_{2,i} + (1-p''_i) N_{1,i} + (f_{k_2}^0(p'')_i - p''_i) N_{2,i} = \\ \nonumber &= \sum_{i=1}^k N_{1,i}^2 + N_{2,i}^2 + N_{1,i}N_{2,i} + p_i N_{1,i} + g_{k_1}^1(p)_i N_{2,i}, \end{align} where for $p \in \mathcal{P}_{k_1+k_2}^k$ given by $p:=1-p''$ we have $f_{k_2}^0(p'') - p'' = g_{k_1}^1(p)$. Furthermore, it is easy to check that $pos_{0,1}(f_{k_2}^0 (p'')) = pos_{1,k_2+1}(p)$ holds.

    Transformation of individual factors appearing in $\delta_{p''}^1(q)$ of $N_{k_0,k_1,k_2}(q)$ is, in complete analogy with proof of Lemma \ref{lem_1}, carried out together with transformation of the corresponding term in the denominator of the left-hand side of \eqref{lem_3_eq1}, i.e. we look at how $\frac{1-\delta_{p''_{i} - p''_{i+1},-1}q^{N_{1,i} - N_{1,i+1}}}{(q)_{N_{1,i} - N_{1,i+1}}}$ transforms under \eqref{lem_3_eq2}. It can be easily checked that we get (together with special case of $i=k$, which produces $1-q^{N_{1,k}}$ when $p''_{k}=1$, i.e. when $p_{k}'=0$): \begin{align} \label{lem_3_eq4} & \frac{1-\delta_{p''_{i} - p''_{i+1},-1}q^{N_{1,i} - N_{1,i+1}}}{(q)_{N_{1,i} - N_{1,i+1}}} \xrightarrow{\eqref{lem_3_eq2}} \\ \nonumber & \xrightarrow{\eqref{lem_3_eq2}} \frac{1-\delta_{p''_{i} - p''_{i+1},1}q^{N_{1,i} - N_{1,i+1}}}{(q)_{N_{1,i} - N_{1,i+1}}} = \frac{1-\delta_{p_{i} - p_{i+1},-1}q^{N_{1,i} - N_{1,i+1}}}{(q)_{N_{1,i} - N_{1,i+1}}}, \end{align} with summation conditions $N_{1,1} \geq \cdots \geq N_{1,k} \geq 0$ remaining unchanged.

    Finally, after taking into account \eqref{lem_3_eq3} and \eqref{lem_3_eq4}, we get \eqref{lem_3_eq1}. \qedhere \end{proof}

    \begin{corollary} \label{cor_proof} Taken together, Lemmas \ref{lem_0}, \ref{lem_1}, \ref{lem_2}, and \ref{lem_3} give the proof of Theorem \ref{thm_solution}. \end{corollary}

\end{document}